\documentclass[pagebackref,colorlinks,citecolor=blue,linkcolor=blue,urlcolor=blue,filecolor=blue]{article}
\pdfoutput=1
\usepackage[margin=2.5cm]{geometry}
\usepackage{tikz}
\usetikzlibrary{calc}
\usepackage{float}
\usepackage{amsmath}
\usepackage{amsfonts}
\usepackage{amsthm}
\usepackage{enumitem}
\usepackage{dsfont}
\usepackage{amssymb}
\usepackage{pifont}
\usepackage{mathtools}
\usepackage{comment}
\usepackage{graphicx} 
\usepackage{float}
\usetikzlibrary{matrix}

\usepackage{pinlabel} 

\usepackage{hyperref}

\newtheorem{thm}{Theorem}[section]
\newtheorem{lem}[thm]{Lemma}
\newtheorem{prop}[thm]{Proposition}

\newtheorem{cor}[thm]{Corollary}
\theoremstyle{definition}\newtheorem{definition}[thm]{Definition}

\theoremstyle{definition}

\newtheorem{defn}[thm]{Definition}

\newtheorem*{claim*}{Claim}
\newtheorem*{quest*}{Question}
\newtheorem*{remark*}{Remark}
\newtheorem*{fact*}{Fact}

\newcommand{\Z}{\ensuremath{\mathbb{Z}}}
\newcommand{\Q}{\ensuremath{\mathbb{Q}}}
\newcommand{\R}{\ensuremath{\mathbb{R}}}

\title{Secondary representation stability and the ordered configuration space of the once-punctured torus}
\author{Nicholas Wawrykow}
\date{}
\begin{document}

\maketitle
\begin{abstract}
In this paper we study stability patterns in the homology of the ordered configuration space of the once-punctured torus. In the last decade Church \cite{church2012homological} and Church--Ellenberg--Farb \cite{church2015fi} proved that the homology groups of the ordered configuration space of a connected noncompact orientable manifold stabilize in a representation theoretic sense as the number of points in the configuration grows, with respect to a map that adds each new point ``at infinity." Miller and Wilson \cite{miller2019higher} proved that there is a \emph{secondary representation stability} pattern among the unstable homology classes, with respect to adding a pair of orbiting points ``near infinity." This pattern is formalized by considering sequences of homology classes as FIM$^{+}$-modules. We prove that, as FIM$^{+}$-modules, the sequence of ``new" homology generators in the $n^{th}$ homology of the ordered configuration space of $2n-2$ points on the once-punctured torus is neither ``free" nor ``stably zero." We also show that this sequence is generated by homology classes on at most $4$ points. Our proof uses Pagaria's work on the Betti numbers of the ordered configuration space of the torus \cite{pagaria2020asymptotic} to calculate the growth rate of the Betti numbers of the ordered configuration space of the once-punctured torus. Our computations are the first to demonstrate that secondary representation stability is a non-trivial phenomenon in positive-genus surfaces.

\end{abstract}

\section{Introduction}

For a topological space $X$, let 
\[
F_{n}(X):=\{(x_{1}, \dots, x_{n})|x_{i}\in X, x_{i}\neq x_{j}\text{ if }i\neq j\}\subseteq X^{n}
\]
denote the ordered configuration space of $n$ distinct points on $X$. When $X=\R^{d}$, the homology groups of $F_{n}(\R^{d})$ are isomorphic to the space of $n$-ary operations of the operad $\mathcal{P}ois^{d}(n)$; for more see, for example, Sinha's expository paper \cite{sinha2006homology}. For most other manifolds explicit descriptions of the homology groups of their ordered configuration spaces are unknown. In his recent paper, Pagaria \cite[Corollary 2.9]{pagaria2020asymptotic} proved that for $k\ge 3$, the $k^{\text{th}}$ Betti number of the ordered configuration space of the torus is a polynomial of degree $2k-2$ in the number of marked points, and that the $0^{\text{th}}$, $1^{\text{st}}$, and $2^{\text{nd}}$ Betti numbers are polynomials of degree $0$, $1$, $3$, respectively, in the number of marked points. We use that result to prove

\begin{prop}\label{theTheorem}
Let $T^{\circ}$ denote the once-punctured torus. Then, for $k\ge 3$, the $k^{\text{th}}$ Betti number of $F_{n}(T^{\circ})$ is a polynomial in $n$ of degree $2k-2$. For $k=0,1,2$, the $k^{\text{th}}$ Betti number of $F_{n}(T^{\circ})$ is a polynomial in $n$ of degree $0,1,3$, respectively.
\end{prop}

If $X$ is an noncompact manifold such as the once-punctured torus, the ordered configuration spaces have additional structure: if $\dim(X)=d$, there is an embedding
\[
e:X\sqcup \R^{d}\hookrightarrow X. 
\]
Figure \ref{embeddingfig} provides an example of such an embedding (e.g., see Kupers and Miller \cite[Lemma 2.4]{kupers2015improved}).

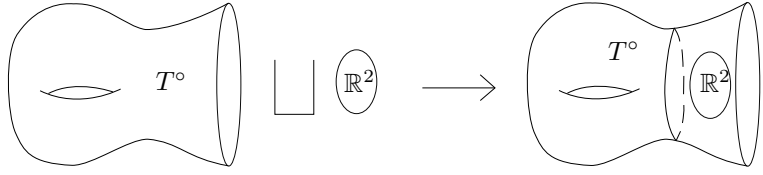
\begin{figure}[h]
\centering
\labellist
\pinlabel {\fontsize{10}{40} $\R^2$} [c] at 284 70
\pinlabel {\fontsize{10}{40} $T^{\circ}$} [c] at 130 70
\pinlabel {\fontsize{10}{40} $\R^2$} [c] at 575 70
\pinlabel {\fontsize{10}{40} $T^{\circ}$} [c] at 500 95
\endlabellist
\includegraphics[width = 10cm]{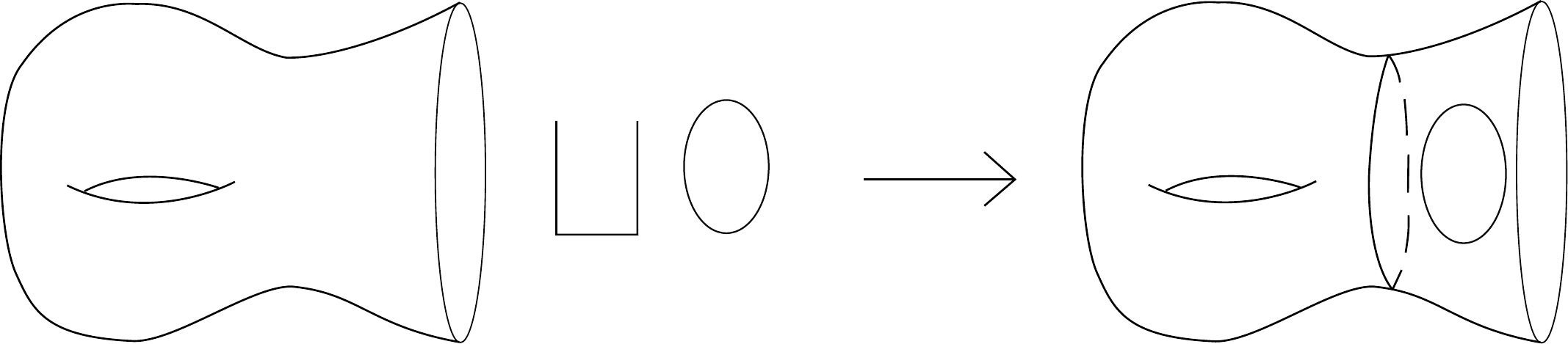}
\caption{An embedding $e:T^{\circ}\sqcup\R^{2}\hookrightarrow T^{\circ}$}
\label{embeddingfig}
\end{figure}

The embedding induces an inclusion of ordered configuration spaces
\[
\iota:F_{n-1}(X)\hookrightarrow F_{n}(X)
\]
\[
\iota(x_{1}, \dots, x_{n-1})\mapsto (e(x_{1}), \dots, e(x_{n-1}), e(0)),
\]
where $0$ denotes the origin in $\R^{d}$. See Figure \ref{firstinclusion}. Thus, $\iota$ maps a configuration of $n-1$ points in $X$ to its image under $e$, and it adds a new point corresponding to the image of the origin in $\R^{d}$ under $e$.

\begin{figure}[h]
\centering
\includegraphics[width = 10cm]{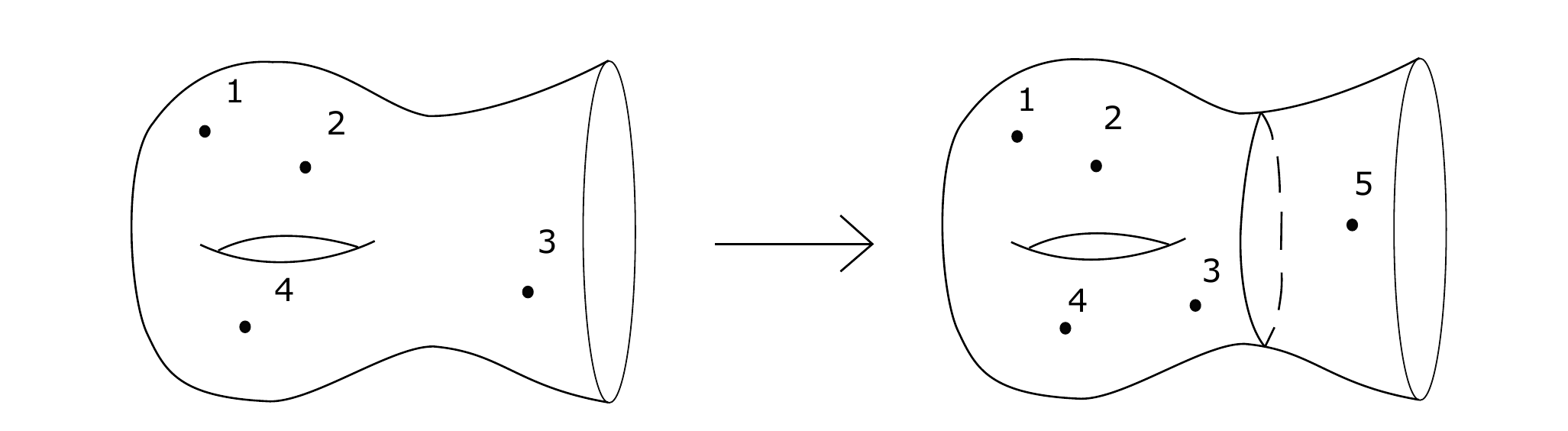}
\caption{The image of a point in $F_{4}(T^{\circ})$ under $\iota$}
\label{firstinclusion}
\end{figure}

The inclusion $\iota$ is an example of a more general result: the embedding $e$ induces inclusions on the product of the ordered configuration spaces of $X$ and $\R^{d}$
\[
F_{n}(X)\times F_{m}(\R^{d})\hookrightarrow F_{n+m}(X).
\]

In this paper we only consider one of these maps, which we will call $\iota'$ 
\[
\iota':F_{n-2}(X)\times F_{2}(\R^{d})\to F_{n}(X)
\]
where
\[
\iota'((x_{1}, \dots, x_{n-2}), (x'_{n-1}, x'_{n})) = (e(x_{1}), \dots, e(x_{n-2}), e(x'_{n-1}), e(x'_{n})).
\]

Figure \ref{doubleinclude} provides an example of $\iota'$.

\begin{figure}[H]
\centering
\includegraphics[width = 12cm]{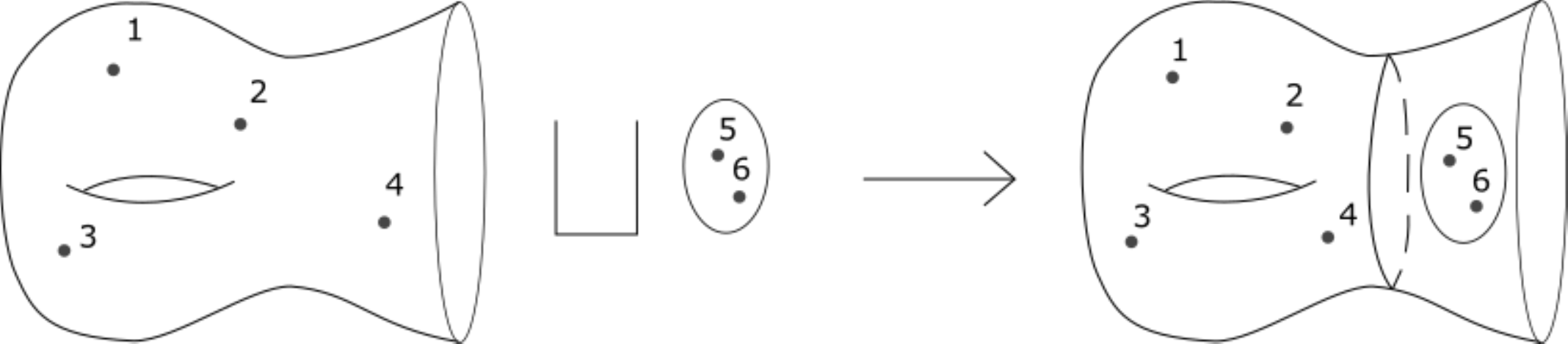}
\caption{The image of a point in $F_{4}(T^{\circ})\times F_{2}(\R^{2})$ under $\iota'$}
\label{doubleinclude}
\end{figure}

The inclusions $\iota$ and $\iota'$ induce maps on homology:
\[
\iota_{*}:H_{k}(F_{n-1}(X))\to H_{k}(F_{n}(X))\indent\text{and}\indent\iota'_{*}:H_{k-1}(F_{n-2}(X))\otimes H_{1}(F_{2}(\R^{d}))\to H_{k}(F_{n}(X)).
\]

The symmetric group $S_{n}$ acts on the configuration space $F_{n}(X)$ by permuting the coordinates, and this induces an action of $S_{n}$ on $H_{k}(F_{n}(X);\Z)$ for all $k$. Thus, we can view the homology groups of the ordered configuration space on $X$ as symmetric group modules. When $n$ is large with respect to $k$, the homology groups of the ordered configuration space of a noncompact manifold $X$ are representation stable in the following sense.

\begin{thm}
(Church--Ellenberg--Farb \cite[Theorem 6.4.3]{church2015fi} in the orientable case and Miller--Wilson \cite[Theorem 3.12]{miller2019higher} in the general case) Let $X$ be a connected noncompact $d$-manifold with $d\ge 2$. For $k\le\frac{n-1}{2}$,
\[
\Z[S_{n}]\cdot \iota_{*}(H_{k}(F_{n-1}(X);\Z))=H_{k}(F_{n}(X);\Z).
\]
\end{thm}

We define another map, which we also denote $\iota'_{*}$, that leads to a notion of secondary representation stability
\[
\iota'_{*}:H_{k-1}(F_{n-2}(X))\to H_{k}(F_{n}(X)).
\]
This map can viewed as a restriction of the other map $\iota'_{*}$, and is obtained by pairing a homology class in $H_{k-1}(F_{n-2}(X))$ with the class in $H_{1}(F_{2}(\R^{d}))$ corresponding to the point $n$ orbiting the point labeled $n-1$ counterclockwise. Miller and Wilson \cite{miller2019higher} were able to show that if $n$ is large with respect to $k$, the homology groups of the ordered configuration space of a noncompact manifold were secondary representation stable in the sense that they satisfied the following theorem.

\begin{thm}\label{jjThm}
(Miller--Wilson \cite[Theorem 1.2]{miller2019higher}) Let $X$ be a connected noncompact finite type $d$-manifold, with $d\ge 2$. There is a function $r:\Z_{\ge 0}\to \Z_{\ge 0}$ tending to infinity such that for $k\le \frac{n-1}{2}+r(n)$,
\[
\Q[S_{n}]\cdot\left(\iota_{*}(H_{k}(F_{n-1}(X);\Q))+\iota'_{*}(H_{k-1}(F_{n-2}(X), \Q))\right)= H_{k}(F_{n}(X); \Q).
\]
\end{thm}

Combining first and second order representation stability we see that there is a range where the $k^{th}$ homology group of the ordered configuration space of $n$ points on $X$ is determined by the $(k-1)^{th}$ homology group of $n-2$ points on $x$ and the $k^{th}$ homology group of $n-1$ points on $X$. See Figure \ref{fig}. In Section \ref{stab} we will see that secondary representation stability can be thought of as a condition on the generation degree of certain ``FIM$^{+}$-modules." \\

\begin{figure}[h]
\centering
\labellist
\pinlabel {\fontsize{10}{40} $k = n$} [c] at 180 200
\pinlabel {\fontsize{10}{40} $k = \frac{n}{2}$} [c] at 380 195
\pinlabel {\fontsize{10}{40} homological degree $k$} [c] at -100 100
\pinlabel {\fontsize{10}{40} Number of points in the ordered configuration space $n$} [c] at 200 -15
\pinlabel {\fontsize{10}{40} homology vanishes} [c] at 75 160
\pinlabel {\fontsize{10}{40} First order representation stable range} [c] at 520 10
\pinlabel {\fontsize{10}{40} FIM$^{+}$-module} [c] at 520 70
\pinlabel {\fontsize{10}{40} Second order representation stable range} [c] at 520 130
\endlabellist
\includegraphics[width = 10cm]{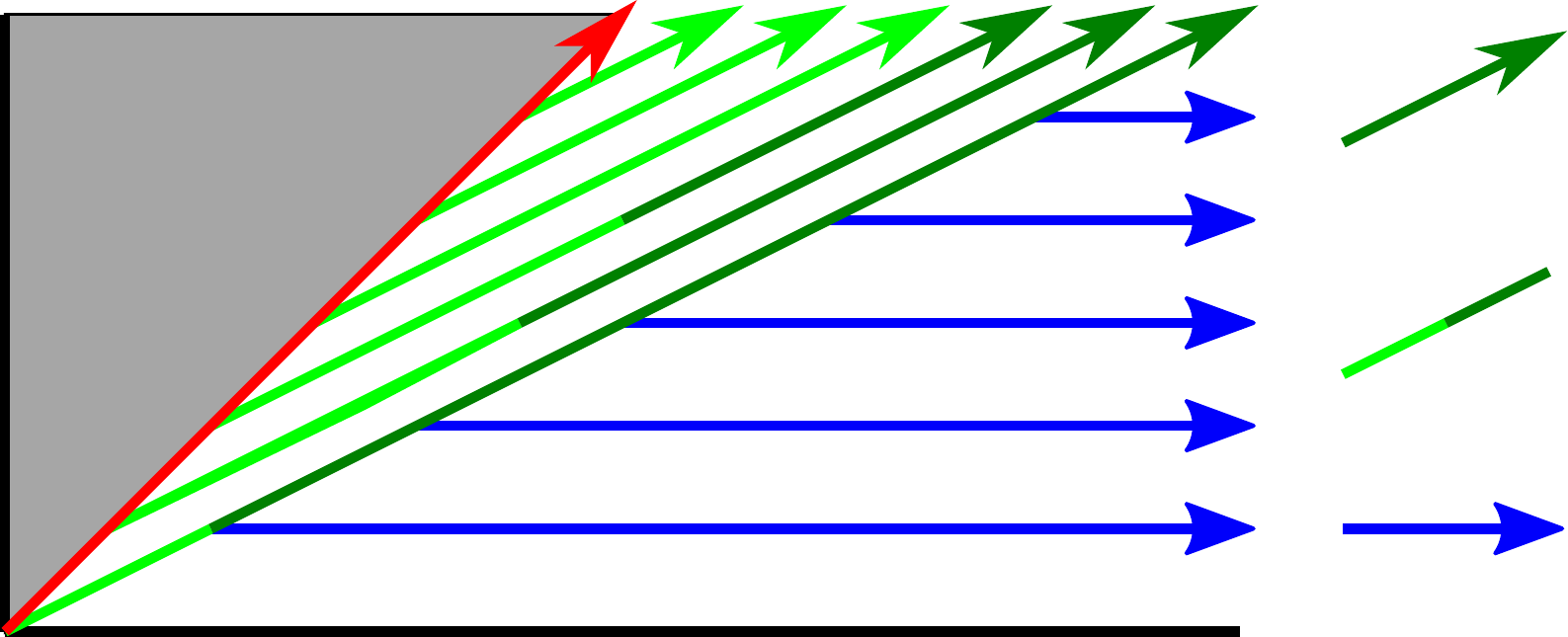}
\caption{First and second order representation stable ranges for surfaces. Compare with \cite[Figure 10]{miller2019higher}}
\label{fig}
\end{figure}
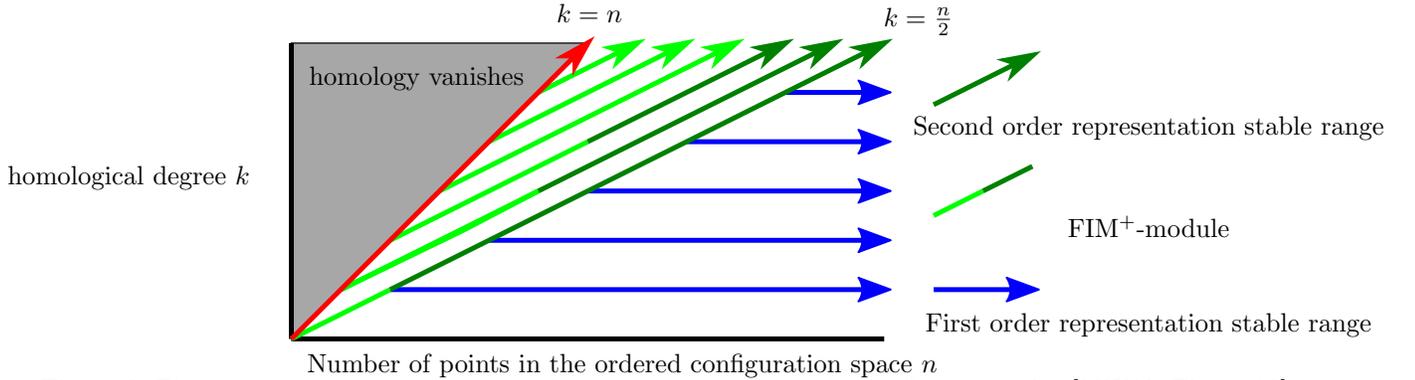

Miller and Wilson also calculated which parts of $H_{k}(F_{n}(X); \Q)$ came from just the $\Q[S_{n}]$-span of the image of $\iota'_{*}$ in several cases.

\begin{prop}
\label{r2structure}
(Miller--Wilson \cite[Proposition 3.33]{miller2019higher}) 
\[
\left(\Q[S_{2n}] \cdot\iota'_{*}(H_{n-1}(F_{2n-2}(\R^{2});\Q))\right)/ \left(\Q[S_{2n}]\cdot\iota_{*}(H_{n}(F_{2n-1}(\R^{2});\Q))\right) \cong\bigoplus_{\lambda\in D_{2n}}V_{\lambda}
\]
as $S_{2n}$ representations, where $V_{\lambda}$ is the irreducible $S_{2n}$-representation associated to the partition $\lambda$, and $\lambda$ is in $D_{2n}$, if and only if when the associated Young diagram is cut in two along the upper staircase, the resultant two skew subdiagrams are symmetric under reflection in the line of slope $-1$. 
\end{prop}

\begin{prop}\label{someFIMmods}
\label{trivialstuff}
(Miller--Wilson \cite[Proposition 3.35]{miller2019higher}) Let $X$ be a connected noncompact surface. If $X$ is not orientable or of genus greater than zero, then
\[
H_{0}(F_{0}(X);\Q) \cong \Z,
\]
and for $n>0$
\[ 
\left(\Q[S_{2n}] \cdot\iota'_{*}(H_{n-1}(F_{2n-2}(X);\Q))\right)/ \left(\Q[S_{2n}]\cdot\iota_{*}(H_{n}(F_{2n-1}(X);\Q))\right) \cong 0.
\]
\end{prop}

In the first example, the structure of the homology groups had previously been determined, and, as we will see, this quotient is a ``free FIM$^{+}$-module"; in the second case the classes of interest stabilize to zero. This raises the question: are there spaces such that the homology arising solely from the $\Q[S_{n}]$-span of $\iota_{*}'$ of their ordered configuration space are secondary representation stable, while being neither free as FIM$^{+}$-modules nor eventually zero?

We answer this in the affirmative. The sequence of $n^{\text{th}}$ homology group generators of the the ordered configuration space of $2n-2$ points on the once-punctured torus $T^{\circ}$ is secondary representation stable, while being neither ``free" nor eventually zero. 

\begin{thm}
\label{mytheorem}
Let $T^{\circ}$ denote the once-punctured torus. For $n\ge 2$,t
\[
\Q[S_{2n-2}]\cdot\iota_{*}(H_{n}(F_{2n-3}(T^{\circ});\Q))\neq H_{n}(F_{2n-2}(T^{\circ});\Q).
\]
For $n\ge4$,
\[
\Q[S_{2n-2}]\cdot\left(\iota_{*}(H_{n}(F_{2n-3}(T^{\circ});\Q))+\iota'_{*}(H_{n-1}(F_{2n-4}(T^{\circ}), \Q))\right) = H_{n}(F_{2n-2}(T^{\circ});\Q).
\]
Moreover, the sequence
\[
\left\{\left(\Q[S_{2n-2}]\iota'_{*}(H_{n-1}(F_{2n-4}(T^{\circ}), \Q))\right)/\left(\Q[S_{2n-2}]\iota_{*}(H_{n}(F_{2n-3}(T^{\circ}), \Q))\right)\right\}_{n\in \Z_{\ge 1}}
\]
is not a free FIM$^{+}$-module.
\end{thm}

\subsection{Acknowledgements}
Jennifer Wilson was of utmost help in understanding \cite{miller2019higher} and writing this paper. I would like to thank John Wiltshire-Gordon for sharing his computations of small-degree Betti numbers of $F_{n}(T^{\circ})$ with Miller and Wilson; these computations were an indirect inspiration for this paper. I would also like to thank Jeremy Miller, Ben Knudsen, Zachary Himes, and Bradley Zykoski for their insightful comments on this paper.

\section{FB-mod, FI-mod, and FIM$^{+}$-mod}
\label{stab}

We introduce the categories FB-mod, FI-mod, and FIM$^{+}$-mod. We use these categories to formalize the notion of second order representation stability. One could rephrase the definitions and results for FI-mod and FIM$^{+}$-mod in the language of (skew)-twisted commutative algebras; for some examples of this see \cite{sam2012introduction} and \cite{nagpal2019noetherianity}.

\begin{definition}
Let \emph{FB} be the category whose objects are all finite (possibly empty) sets and whose morphisms are bijective maps.
\end{definition}

Every finite set is isomorphic to $[n] := \{1, \dots, n\}$ for some $n$; a choice of such isomorphisms for all finite sets provides an equivalence between FB and its full subcategory that has one set $[n]$ for each $n\in \Z_{0}^{+}$.

\begin{defn}
The category of \emph{FB-modules} over a ring $R$ has covariant functors from FB to the category of $R$-modules as objects, and natural transformations between these functors as morphisms.
\end{defn}

For an FB-module $W$ and a finite set $S$, let $W_{S}$ denote the corresponding $R$-module. When $S$ is the set $[n]$, we write $W_{n}$ for $W_{[n]}$. Each $W_{n}$ carries an action of $S_{n}$ arising from the equivalence $S_{n}\simeq \text{End}_{\text{FI}}([n])$; therefore, we can view $\{W_{n}\}$ as a sequence of symmetric group representations.

\begin{definition}
Let \emph{FI} be the category whose objects are all finite (possibly empty) sets and whose morphisms are injective maps.
\end{definition}

Just as for FB, there is an equivalence between FI and its full subcategory that has one set $[n]$ for each $n\in \Z_{\ge0}^{+}$.

\begin{definition}
The category of \emph{FI-modules} over the ring $R$ has covariant functors from FI to the category of $R$-modules as objects, and natural transformations between these functors as morphisms. Similarly, objects of the the category of \emph{FI-(homotopy)-space} are covariant functors from FI to the (homotopy)-category of topological spaces, and the morphisms are natural transformations between these functors.
\end{definition}

Much like an FB-module, an FI-module is a sequence of symmetric group representations; however, there are nontrivial maps between these representations. Let $V$ be an FI-module, if $\iota_{n,m}$, $n<m$, denotes the standard inclusion of $[n]$ into $[m]$, then $(\iota_{n, m})_{*}:V_{n}\to V_{m}$ must be $S_{n}$-equivariant; moreover, $(\iota_{n, m})_{*}(V_{n})$ must be invariant under the action of $S_{[m]-[n]}$.

\begin{definition}
An \emph{FI-submodule} $V'$ of an FI-module $V$, is a sequence of symmetric group representations $\{V'_{n}\}\subseteq \{V_{n}\}$ that is closed under the action of the FI-morphism.
\end{definition}

We want to build an FI-module out of an FB-module. Let  
\[
M(-)_{n} :=R[\text{Hom}_{\text{FI}}([-],[n]);
\]
letting $n$ vary over the nonnegative integers makes this an FI-module. If $W_{d}$ is an $S_{d}$-representation, we can define $M(W_{d})$ by setting
\[
M(W_{d})_{n}:=W_{d}\otimes_{R[S_{d}]}M(d)_{n}.
\]
Letting $n$ vary over the nonnegative integers one can check that $M(W_{d})$ is an FI-module.

Since an FB-module is a sequence of symmetric group representations, we can apply $M(-)$ to every degree of an FB-module $W$; this is a functor from FB-mod to FI-mod.

\begin{definition}
Given an FB-module $W$, let $M(W)$ denote the associated \emph{free FI-module}
\[
M(W):=\bigoplus_{d\ge 0}M(W_{d}).
\]
\end{definition}

\begin{definition}
An FI-module $V$ is \emph{generated} by a set $S\subseteq \coprod_{n\ge 0} V_{n}$ if $V$ is the smallest FI-submodule containing $S$. If there is some finite set $S$ that generates $V$, then $V$ is \emph{finitely generated}. If $V$ is generated by $\coprod_{0\le n\le d}V_{n}$, then $V$ is \emph{generated in degree $\le d$}. 
\end{definition}

We want to recover a generating set for an FI-module $V$, preferably a minimal one. Such a generating set consists of subrepresentations of $V_{n}$, for all $n$, not arising from the FI-structure in smaller degrees. We use the language of FI-homology to formalize this.

\begin{definition}
The \emph{zeroth FI-homology group} of an FI-module $V$ in degree $n$, denoted \emph{$H_{0}^{\emph{FI}}(V)_{n}$}, is the set of $S_{n}$-representations in $V_{n}$ not arising from the image of $V_{n-1}$ under the maps induced by all inclusions $f:[n-1]\hookrightarrow [n]$:
\[
H_{0}^{\text{FI}}(V)_{n}:= V_{n}\backslash\bigoplus_{f:[n-1]\hookrightarrow [n]}f_{*}V_{n-1}.
\]
\end{definition}

Note that $H_{0}^{\text{FI}}(V)_{n}$ is an $S_{n}$-representation, so $\{H_{0}^{\text{FI}}(V)\}_{n}$ is a sequence of symmetric group representations, i.e., an FB-module. Thus, we can think of $H_{0}^{\text{FI}}(-)$ as a functor from FI-mod to FB-mod.

\begin{definition}
A \emph{based set} $S_{*}$ is a set with a distinguished element $*\in S_{*}$, the \emph{basepoint}. A map of based sets $f:S_{*}\to T_{*}$ takes $*\in S_{*}$ to $*\in T_{*}$. We define \emph{FI\#} to be the category whose objects are finite  based sets and whose morphisms are maps of based sets that are injective away from the basepoint, i.e., if $f:S_{*}\to T_{*}$ is an FI\#-morphism, then $|f^{-1}(t)|\le 1$ for all $t\in T_{*}$, $t \neq *$.
\end{definition}

\begin{definition}
The category of \emph{FI\#-modules} over the commutative ring $R$ has covariant functors from FI\# to the category of $R$-modules as objects, and natural transformations as morphisms. We can similarly define the category of \emph{FI\#-homotopy-spaces}.
\end{definition}

An FI\#-module can be viewed as an FI-module by forgetting the morphisms in FI\# that are not injections. Similarly, an FI\#-module can also be seen to be an FI$^{op}$-module by considering only surjective morphisms in FI\#. Moreover, FI\# is equivalent to FI\#$^{op}$.

\begin{thm}\label{equiv}
(Church--Ellenberg--Farb \cite[Theorem 4.1.5]{church2015fi})  The category of FI\#-modules is equivalent to the category of FB-modules via the equivalence of categories
\[
M(-):\text{FB-Mod}\leftrightarrows \text{FI\#-Mod} :H_{0}^{FI}(-).
\]
Thus, every FI\#-module $V$ is of the form $\oplus^{\infty}_{n=0}M(H_{0}^{FI}(V)_{n})$.
\end{thm}

Recall the definition of the ordered configuration space of $n$ points on a open manifold $X$ and the inclusion map $\iota:F_{n-1}(X)\hookrightarrow F_{n}(X)$ given in the introduction. The inclusion $\iota$ is well-behaved up to homotopy with respect to the symmetric group action, making $F_{*}(X)$ an FI-homotopy-space. Taking the homology groups of these ordered configuration spaces gives us a sequence of FI-modules: for fixed $k\ge 0$, $H_{k}(F_{*}(X))$ is an FI-module. We can say even more, namely that the forgetful map 
\[
\pi:F_{n}(X)\to F_{n-1}(X),
\]
given by forgetting the last coordinate
\[
\pi(x_{1}, \dots, x_{n})=(x_{1}, \dots, x_{n-1}),
\]
is well behaved with respect to the symmetric group action, and $F_{*}(X)$ is an FI\#-homotopy-space. Fixing $k$, this makes $H_{k}(F_{*}(X))$ an FI\#-module. By Theorem \ref{equiv}, the zeroth FI-homology of these homology groups, $H^{FI}_{0}\left(H_{k}(F_{*}(X))\right)$, is an FB-module. 

For a thorough overview of FI and FI-mod see \cite{church2015fi} or \cite{wilson2018introduction}.

\begin{definition}
A \emph{matching} of a set $A$ is a set of disjoint $2$-element subsets of $A$, and a matching is said to be \emph{perfect} if the union of these subsets is $A$. 
\end{definition}

The category FI ignores the data of the complement of the image of a morphism; by insisting on a perfect matching on the complement of the image we get the category FIM.

\begin{definition}
Let \emph{FIM} denote the category whose objects are finite sets and whose morphisms are injective maps $f:A\hookrightarrow B$ along with a perfect matching on $B\backslash f(A)$.
\end{definition}

Morphisms between two objects $A, B$ of FIM exist only when $|A|$ and $|B|$ have the same parity. If there are morphisms from $A$ to $B$ and we insist on an ordering of the perfect matching, then the symmetric group $S_{m}$, $m=\frac{|B|-|A|}{2}$, acts on an ordered perfect matching $B_{1}, \dots, B_{m}$ on the complement of the image of $A$ in $B$ by permuting the ordering:
\[
\sigma\cdot(B_{1}, \dots, B_{m})= (B_{\sigma(1)}, \dots, B_{\sigma(m)}).
\]

This inspires the definition of FIM$^{+}$, a category enriched over $R$-mod.

\begin{definition}
Let \emph{FIM$^{+}$} be the category whose objects are finite sets and whose module of morphisms $f$, consist of injective maps with an ordered perfect matching on the complement quotiented by a signed symmetric group action:
\[
\frac{R\big\langle (f:A\to B, B_{1}, \dots, B_{m})\big|f \text{ is injective, } |B_{i}|=2, B=\text{Im}(f)\sqcup B_{1} \cdots \sqcup B_{m}\big\rangle}{\langle (f, B_{1}, \dots, B_{m})=\text{sign}(\sigma)(f, B_{\sigma(1)}, \dots, A_{\sigma(m)}) \text{ for all } \sigma\in S_{m}\rangle}.
\]
\end{definition}

\begin{definition}
The category of \emph{FIM$^{+}$-modules} over the ring $R$ has covariant functors from FIM$^{+}$ to the category of $R$-modules as objects, and natural transformations between these functors as morphisms.
\end{definition}

\begin{definition}
An \emph{FIM$^{+}$-submodule} $W'$ of an FIM$^{+}$-module $W$ is a sequence of symmetric group representations $W'_{n}\subseteq W_{n}$ that is closed under the action of the FIM$^{+}$-morphism.
\end{definition}

\begin{definition}
An FIM$^{+}$-module $W$ is \emph{generated} by a set $S\subseteq \coprod_{n\ge 0} W_{n}$ if $W$ is the smallest FIM$^{+}$-submodule containing $S$. If there is some finite set $S$ that generates $W$, then $W$ is \emph{finitely generated}. If $W$ is generated by $\coprod_{0\le n\le d}V_{n}$, then $W$ is \emph{generated in degree $\le d$}. 
\end{definition}

Let
\[
M^{\text{FIM}^{+}}(-)_{n}:= R[\text{Hom}_{\text{FIM}^{+}}([-], [n])].
\]
By letting $n$ vary we see that this takes in an integer and produces an FIM$^{+}$-module. We can extend this to a functor from the category of $S_{d}$-representations by setting
\[
M^{\text{FIM}^{+}}(W_{d})_{n}:=W_{d}\otimes_{R[S_{d}]}M^{\text{FIM}^{+}}(d)_{n}.
\]

By taking a sequence of symmetric group representations $\{W_{d}\}$ we can define free FIM$^{+}$-modules.

\begin{definition}
Given an FB-module $W$, let $M^{\text{FIM}^{+}}(W)$ denote the associated \emph{free FIM$^{+}$-module}
\[
M^{\text{FIM}^{+}}(W):=\oplus_{d\ge 0}M^{\text{FIM}^{+}}(W_{d}).
\]
\end{definition}

Note that
\[
\dim\left(M^{\text{FIM}^{+}}(0)_{2n}\right)=\frac{(2n)!}{n!2^{n}},
\]
and
\[
\dim\left(M^{\text{FIM}^{+}}(W_{d})_{2n}\right)=\binom{2n}{d}\dim(W_{d})\dim\left(M^{\text{FIM}^{+}}(0)_{2n-d}\right).
\]

When $X$ is a noncompact manifold, we let $\mathcal{W}_{i}^{X}(n)$ denote the sequence of minimal generators of the homology groups of its ordered configuration space
\[
\mathcal{W}_{i}^{X}(n):=H^{\text{FI}}_{0}\left(H_{\frac{n+i}{2}}(F(X);\Q)\right)_{n}.
\]

By setting non-integer homology to be $0$, it follows that there is an $S_{n}$-action on $\mathcal{W}_{i}^{X}(n)$ for all $n\ge 0$. Thus, $\mathcal{W}_{i}^{X}$ is an FB-module. Miller and Wilson proved that it also has an FIM$^{+}$-module structure \cite{miller2019higher}. Additionally, they proved that as an FIM$^{+}$-module it is secondary representation stable in the sense that it satisfies the following theorem, a reformulation of Theorem \ref{r2structure}.

\begin{thm}\label{secondStab}
(Miller--Wilson \cite[Theorem 1.4]{miller2019higher}) If $K$ is a field of characteristic zero and $X$ is a connected noncompact $d$-manifold of finite type and $d\ge 2$, then, for each $i\ge 0$, the sequence of minimal generators
\[
\mathcal{W}_{i}^{X}(n)=H^{\emph{FI}}_{0}\left(H_{\frac{n+i}{2}}(F(X);K)\right)_{n}
\]
is finitely generated as an FIM$^{+}$-module.
\end{thm}

To prove this theorem, Miller and Wilson used the complex of injective words, the arc resolution spectral sequence, and a Noetherianity result of Nagpal, Sam, and Snowden \cite{nagpal2019noetherianity}. In addition to proving this theorem, they computed some explicit examples of secondary representation stability.

\begin{prop}
(Miller--Wilson \cite[Proposition 3.33]{miller2019higher}) 
\[
\mathcal{W}_{0}^{\R^{2}}(2n)\cong M^{\emph{FIM}^{+}}(0)_{2n} \cong\bigoplus_{\lambda\in D_{2n}}V_{\lambda}
\]
where $V_{\lambda}$ is the irreducible $S_{2n}$ representation corresponding to the partition $\lambda$, and $\lambda$ is in $D_{2n}$ if and only if when the associated Young diagram is cut in two along the upper staircase, the resultant two skew subdiagrams are symmetric under reflection in the line of slope $-1$.
\end{prop}

\begin{prop}\label{someFIMmods}
(Miller--Wilson \cite[Proposition 3.35]{miller2019higher}) Let $X$ be a connected, noncompact surface. If $X$ is not orientable or of genus greater than zero, then
\[
\mathcal{W}_{0}^{X}(0) \cong \Z\indent\text{ and }\indent\mathcal{W}_{0}^{X}(2n)\cong 0\text{ for }n>0.
\]
\end{prop}

These explicit examples of secondary representation stability are edge cases. When $X=\R^{2}$ the FIM$^{+}$-module $\mathcal{W}_{i}^{\R^{2}}(2n)$ is a free FIM$^{+}$-module for all $i\ge 0$ and all the homology groups of $F_{*}(\R^{2})$ are known. For a general connected surface $X$, it was known that $\mathcal{W}_{0}^{X}(0) = H^{\text{FI}}_{0}(H_{0}(F(X)))_{0}=H_{0}(F_{0}(X))\cong \Z$. For $n>0$, $\mathcal{W}_{0}^{X}(2n)=0,$ i.e., stably zero. This is the opposite extreme of $\mathcal{W}_{0}^{\R^{2}}(2n)$ being free. 

All examples of secondary representation stability where the FIM$^{+}$-module structure is known, including the ones above, are either free FIM$^{+}$-modules or stably zero. We seek an example of secondary representation stability that lies between these extremes; namely, a finitely generated FIM$^{+}$-module arising from the homology groups of the ordered configuration space of a surface that is neither free nor stably zero.

\section{Arc Resolution Spectral Sequence}
\label{arcres}

In this section we describe the arc resolution spectral sequence that Miller and Wilson used to prove their secondary representation stability results. We state a few important results on the differentials of this spectral sequence, which we will later use to bound the generation degree of two FIM$^{+}$-modules of the form $\mathcal{W}_{i}^{T^{\circ}}(n)$.

The arc resolution spectral sequence arises from the augmented semi-simplical space $Arc_{*}(F_{*}(M))$, see \cite[Definition 3.6]{miller2019higher} for a definition and \cite[Section 3.2]{miller2019higher} for more details about this spectral sequence. For more about spectral sequences arising from augmented semi-simplical spaces see, for example, Randal-Williams \cite[Section 2.3]{randal2013homological}.

\begin{prop}
\label{arcresprop}
(Miller--Wilson \cite[Proposition 3.10 and Proposition 3.8]{miller2019higher}) Let $M$ be a noncompact connected smooth $d$-manifold $d\ge 2$. There is a spectral sequence called the arc resolution spectral sequence that satisfies:
\[
E^{2}_{p, q}[M](n)\cong\emph{Ind}^{S_{n}}_{S_{p+1}\times S_{n-p-1}} \mathcal{T}_{p+1}\boxtimes H_{0}^{\emph{FI}}(H_{q}(F(M)))_{n-p-1},
\] 
where $\mathcal{T}_{p}$ is a free abelian group of rank $\sum^{p}_{i=0}(-1)^{i}\frac{p!}{i!}$. Moreover, 
\[
E^{\infty}_{p,q}[M](n)=0 \emph{ for }p+q+2 \le n.
\]
See Figure \ref{arcresfig}.
\end{prop}

\begin{figure}[H]
\centering
\begin{tikzpicture} \footnotesize
  \matrix (m) [matrix of math nodes, nodes in empty cells, nodes={minimum width=3ex, minimum height=5ex, outer sep=2ex}, column sep=3ex, row sep=3ex]{
  4    & H_0^{\text{FI}}\Big(H_4(F(M))\Big)_{n}  & 0 &  \text{Ind}_{S_{2}\times S_{n-2}}^{S_n} \mathcal{T}_{2} \boxtimes H_0^{\text{FI}} (H_4( F(M)))_{n-2}   &\text{Ind}_{S_{3}\times S_{n-3}}^{S_n} \mathcal{T}_{3} \boxtimes H_0^{\text{FI}} (H_4( F(M)))_{n-3} & \\  
 3    &  H_0^{\text{FI}}\Big(H_3(F(M))\Big)_{n}  &0 &  \text{Ind}_{S_{2}\times S_{n-2}}^{S_n} \mathcal{T}_{2} \boxtimes H_0^{\text{FI}} (H_3 (F(M)))_{n-2}   &\text{Ind}_{S_{3}\times S_{n-3}}^{S_n} \mathcal{T}_{3} \boxtimes H_0^{\text{FI}} (H_3 (F(M)))_{n-3} & \\  
 2    &  H_0^{\text{FI}}\Big(H_2(F(M))\Big)_{n} & 0&  \text{Ind}_{S_{2}\times S_{n-2}}^{S_n} \mathcal{T}_{2} \boxtimes H_0^{\text{FI}} (H_2( F(M)))_{n-2}    &\text{Ind}_{S_{3}\times S_{n-3}}^{S_n} \mathcal{T}_{3} \boxtimes H_0^{\text{FI}} (H_2( F(M)))_{n-3} &  \\          
1     &  H_0^{\text{FI}}\Big(H_1(F(M))\Big)_{n}   &0    & \text{Ind}_{S_{2}\times S_{n-2}}^{S_n} \mathcal{T}_{2} \boxtimes H_0^{\text{FI}} (H_1( F(M)))_{n-2}  &\text{Ind}_{S_{3}\times S_{n-3}}^{S_n} \mathcal{T}_{3} \boxtimes H_0^{\text{FI}} (H_1( F(M)))_{n-3} & \\             
 0     &  H_0^{\text{FI}}\Big(H_0(F(M))\Big)_{n}  & 0 & \text{Ind}_{S_{2}\times S_{n-2}}^{S_n} \mathcal{T}_{2} \boxtimes H_0^{\text{FI}} (H_0( F(M)))_{n-2}  &\text{Ind}_{S_{3}\times S_{n-3}}^{S_n} \mathcal{T}_{3} \boxtimes H_0^{\text{FI}} (H_0( F(M)))_{n-3}  & \\       
 \quad\strut &   -1  &  0  &  1  & 2  &\\}; 

\draw[thick] (m-1-1.east) -- (m-6-1.east) ;
\draw[thick] (m-6-1.north) -- (m-6-6.north east) ;
\end{tikzpicture}
\caption{The $E^{2}$-page of the arc resolution spectral sequence}
\label{arcresfig}
\end{figure}

The differentials of the arc resolution spectral sequence satisfy a Leibniz rule.

\begin{lem}
\label{Leibniz}
(Miller--Wilson \cite[Lemma 3.15]{miller2019higher}) Let $M$ be a noncompact connected smooth $d$-manifold $d\ge 2$. There exist maps
\[
t^{r}:E^{r}_{p,q}[\R^{d}](S)\otimes E^{r}_{p',q'}[M](T)\to E^{r}_{p+p'+1, q+q'}[M](S\sqcup T)
\]
arising from the embedding of $\text{Arc}_{p}(F_{S}(\R^{d}))\times \text{Arc}_{p'}(F_{T}(M))\to \text{Arc}_{p+p'}(F_{S\sqcup T}(M))$ given in \cite[Definition 3.14]{miller2019higher}. The maps $t^{r}$ satisfy a Leibniz rule with respect to the differentials: if $a\in E^{r}_{p,q}[\R^{d}](S)$ and $b\in E^{r}_{p',q'}[M](T)$, then
\[
d^{r}(t^{r}(a\otimes b))=t^{r}(d^{r}(a)\otimes b)+(-1)^{p+q}t^{r}(a\otimes d^{r}(b)).
\]
\end{lem}

When $n = 2$ or $3$, the group $\mathcal{T}_{n}$ is isomorphic a free group of rank $n-1$, with $\Z$-basis $[1,2]$ and $[[1,2],3], [[1,3],2]$, respectively, by \cite[Theorem 2.33]{miller2019higher} and \cite[Theorem 2.35]{miller2019higher}. These are Reutenauer's basis elements for $\mathcal{L}_{n}$, a subgroup of $\mathcal{T}_{n}$. For more information on $\mathcal{T}_{n}$ and $\mathcal{L}_{n}$ see \cite[Section 2.3]{miller2019higher}. There are explicit formulae for differentials in the arc resolution spectral sequence when applied to Reutenauer's basis elements.

\begin{lem}
\label{bullseye}
(Miller--Wilson \cite[Lemma 3.17]{miller2019higher}) Let $E^{r}_{p,q}[M](n)$ be the arc resolution spectral sequence, and let
\[
L=[[[\cdots[m_{1}, m_{2}], m_{3}],\dots],m_{n}]\in E^{1}_{k-1,0}[M](n).
\]
Then $d^{r}(L)=0$ for $r<k$ and $d^{n}(L)$ is the homology class of $\psi(\cdots \psi(\psi(m_{1}, m_{2}), m_{3}), \dots, m_{n})$ in $E^{n}_{-1, n-1}[M](n)$, where the $m_{i}$ are distinct points in $M$ and $\psi(x,y)$ denotes $y$ orbiting $x$. See Figure \ref{tripbull} for a representative of $\psi(\psi(m_{1}, m_{2}), m_{3}))$ in $F_{3}(T^{\circ})$.
\end{lem}

We can use the Leibniz rule and our knowledge of the differentials from $\mathcal{L}_{n}$ to calculate many of the differentials of the arc resolution spectral sequence and relate them to FIM$^{+}$-morphisms.

\begin{prop}
\label{d2isFIM}
Let $M$ be a noncompact $2$-manifold. The $E^{2}[M]$-page differential from the $1^{\text{st}}$ column to the $-1^{\text{st}}$ column of the arc resolution spectral sequence,
\[
d^{2}:E^{2}_{1,q}[M]\to E^{2}_{-1,q+1}[M],
\]
is the FIM$^{+}$-morphism arising from the standard inclusion $[n-2]\hookrightarrow[n]$.
\end{prop}

\begin{proof}
Note that
\[
E^{2}_{1,q}[M](n) = \text{Ind}^{S_{n}}_{S_{2}\times S_{n-2}}\mathcal{T}_{2}\boxtimes H_{0}^{\text{FI}}\left(H_{q}(F(M))\right)_{n-2}.
\]
Since $M$ is $2$-dimensional,
\[
E^{2}_{1,q}[M](n) = t^{2}\left(E^{2}_{1, 0}[\R^{2}](2)\otimes E^{2}_{-1,q}[M](n-2)\right).
\]
For $p<-1$, all the elements $E_{p,q}[M]$ of the arc resolution spectral sequence are $0$, so
\[
d^{2}(E^{2}_{-1,q}[M](n-2))=0.
\]
Thus Lemma \ref{Leibniz}, shows
\[
d^{2}(E^{2}_{1,q}[M](n))=d^{2}\left(t^{2}(E^{2}_{1, 0}[\R^{2}](2)\otimes E^{2}_{-1,q}[M](n-2))\right)=t^{2}\left(d^{2}(E^{2}_{1, 0}[\R^{2}](2))\otimes E^{2}_{-1,q}[M](n-2) \right)
\]
By Lemma, \ref{bullseye} the image of $E^{2}_{1, 0}[\R^{2}](2)$ under $d^{2}$ corresponds to two points orbiting each other in $H_{1}(F_{2}(\R^{2}))_{2}$. Therefore, the image of
\[
t^{2}\left(d^{2}(E^{2}_{1, 0}[\R^{2}](2))\otimes E^{2}_{-1,q}[M](n-2) \right)\subset E^{2}_{-1, q+1}[M](n)
\]
corresponds to the cup product of the homology of two points orbiting each other in $H_{1}(F_{2}(\R^{2}))_{2}$ and the homology classes of $E^{2}_{1,q}[M](n-2)=H_{0}^{\text{FI}}\left(H_{q}.(F(M))\right)_{n-2}$ in $H_{0}^{\text{FI}}\left(H_{q+1}.(F(M))\right)_{n}$. This is the FIM$^{+}$-morphism arising from the inclusion $[n-2]\hookrightarrow[n]$.
\end{proof}

\section{Ordered Configuration Space of the Once-Punctured Torus}
\label{torus}

We begin this section by using Pagaria's calculation of the growth rate of the Betti numbers of the ordered configuration space of the torus \cite{pagaria2020asymptotic} to calculate the growth rate of the Betti numbers of the ordered configuration space of the once-punctured torus. We use this result and the arc resolution spectral sequence to prove an FIM$^{+}$-module version of Theorem \ref{mytheorem}.

\begin{thm}\label{pagaria}
(Pagaria \cite[Corollary 2.9]{pagaria2020asymptotic}) For $k\ge 3$ the $k^{\emph{th}}$ Betti number of $F_{n}(T)$ is of the form
\[
b_{k}(F_{n}(T))=c_{k}\binom{n}{2k-2}+o(n^{2k-2}),
\]
where $c_{k}$ is a constant greater than or equal to $\binom{2k-3}{k-3}$.

For $k\le 5$ the Betti numbers are
\[
b_{0}(F_{n}(T)) = 1,
\]
\[
b_{1}(F_{n}(T))  = 2n,
\]
\[
b_{2}(F_{n}(T))  = 2\binom{n}{3}+3\binom{n}{2}+n,
\]
\[
b_{3}(F_{n}(T)) = 14\binom{n}{4}+8\binom{n}{3}+2\binom{n}{2},
\]
\[
b_{4}(F_{n}(T)) = 32\binom{n}{6}+74\binom{n}{5}+33\binom{n}{4}+5\binom{n}{3},
\]
\[
b_{5}(F_{n}(T)) = 63\binom{n}{8}+427\binom{n}{7}+490\binom{n}{6}+154\binom{n}{5}+18\binom{n}{4}.
\]
\end{thm}

We prove a similar result for the ordered configuration space of the once-punctured torus by noting that the torus is additive group.

\begin{prop}\label{thetheorem}
Let $T^{\circ}$ denote the once-punctured torus. For $k\ge 3$, the $k^{\emph{th}}$ Betti number of $F_{n}(T^{\circ})$ is a polynomial in $n$ of degree $2k-2$.  For $k\le 5$ the formulae for the $k^{\emph{th}}$ Betti numbers are
\[
b_{0}(F_{n}(T^{\circ})) = 1,
\]
\[
b_{1}(F_{n}(T^{\circ})) = 2n,
\]
\[
b_{2}(F_{n}(T^{\circ}))
=2\binom{n}{3}+5\binom{n}{2},
\]
\[
b_{3}(F_{n}(T^{\circ}))
= 14\binom{n}{4}+18\binom{n}{3},
\]
\[
b_{4}(F_{n}(T^{\circ})) = 32\binom{n}{6}+106\binom{n}{5}+79\binom{n}{4},
\]
\[
b_{5}(F_{n}(T^{\circ})) = 63\binom{n}{8}+490\binom{n}{7}+853\binom{n}{6}+432\binom{n}{5}.
\]
\end{prop}

\begin{proof}
Since $T\simeq \R^{2}/ \Z^{2}$ is an additive group, we can decompose the ordered configuration space of the torus as a product:
\begin{gather*}
F_{n}(T)\simeq T\times F_{n-1}(T^{\circ})\\
(x_{1}, \dots, x_{n})\mapsto x_{1}\times (x_{2}-x_{1}, \dots, x_{n}-x_{1}),
\end{gather*}
where coordinates in $F_{n-1}(T^{\circ})$ are taken modulo $\Z^{2}$, and $x_{1}$ is the location of the puncture in $T^{\circ}$. See, for example, Cohen \cite[Example 2.6]{cohen2010introduction}.

Since Poincare polynomials respect product decompositions, we can write
\begin{gather*}
P(F_{n}(T))=P(T)\times P(F_{n-1}(T^{\circ})) = (1+2t+t^2)P(F_{n-1}(T^{\circ})).
\end{gather*}

These equations can be rearranged to give the Poincare polynomial for $F_{n-1}(T^{\circ})$ in terms of the Poincare polynomial for $F_{n}(T)$, 
\[
P(F_{n-1}(T^{\circ}))=\frac{P(F_{n}(T))}{1+2t+t^2}=\sum^{\infty}_{i=0}(-1)^{i}(i+1)t^{i}P(F_{n}(T)),
\]
where the second equality arises by expanding $(1+2t+t^{2})^{-1}$ as a Taylor series in $t$. 

We have
\[
\sum_{k=0}^{\infty}b_{k}(F_{n-1}(T^{\circ}))t^{k}=\left(\sum^{\infty}_{i=0}(-1)^{i}(i+1)t^{i}\right)\left(\sum_{j=0}^{\infty}b_{j}(F_{n}(T))t^{j}\right) =\sum^{\infty}_{k=0}\sum^{k}_{m= 0}(-1)^{k-m}(k+1-m)b_{m}(F_{n}(T))t^{k}.
\]

This gives us a formula for the Betti numbers of $F_{n-1}(T^{\circ})$ in terms of the Betti numbers for $F_{n}(T)$,
\[
b_{k}(F_{n-1}(T^{\circ}))=\sum^{k}_{m= 0}(-1)^{k-m}(k+1-m)b_{m}(F_{n}(T)).
\]

By replacing $n-1$ with $n$ we see that
\[
b_{k}(F_{n}(T^{\circ}))=\sum^{k}_{m= 0}(-1)^{k-m}(k+1-m)b_{m}(F_{n+1}(T)),
\]
e.g., 
\[
b_{0}(F_{n}(T^{\circ})) = b_{0}(F_{n+1}(T)),
\] 
\[
b_{1}(F_{n}(T^{\circ}))=b_{1}(F_{n+1}(T))-2b_{0}(F_{n+1}(T)).
\]
\[
 b_{2}(F_{n}(T^{\circ}))= b_{2}(F_{n+1}(T))-2b_{1}(F_{n+1}(T))+3b_{0}(F_{n+1}(T)),
 \]
 etc.

Now apply Theorem \ref{pagaria}, which states that $b_{k}(F_{n}(T))$ is a polynomial in $n$ of degree $2k-2$, and for $k=0,1,2$, it is a polynomial of degree $0, 1, 3$. Reindexing from $n$ to $n+1$ does not change this since there cannot be any cancelation in the top degree. For $k\ge 3$, the above calculations prove that $b_{k}(F_{n}(T^{\circ}))$ can be written as a linear combination of $k+1$ polynomials in $n$ of degree $\le2k-2$, and that only one of these polynomials is of degree $2k-2$. When $k=0, 1, \text{or }2$, $b_{k}(F_{n}(T^{\circ}))$ can be expressed as linear combination of $k+1$ polynomials in $n$ of degree $\le 0,1, \text{or } 3$, respectively, with only one polynomial in each of these sums having top degree. Therefore, the top degree term in the sum has nonzero coefficient. Thus, for $k\ge 3$, $b_{k}(F_{n}(T^{\circ}))$ is polynomial in $n$ of degree $2k-2$, and, when $k=0, 1, 2$, it is polynomial in $n$ of degree $0, 1, 3$, respectively.

To get the formula for $b_{k}(F_{n}(T^{\circ}))$ for $k\le5$ apply the formula
\[
b_{k}(F_{n}(T^{\circ}))=\sum^{k}_{m= 0}(-1)^{k-m}(k+1-m)b_{m}(F_{n+1}(T))
\]
to the formulae for the Betti numbers of the ordered configuration space of the torus from Theorem \ref{pagaria}.
\end{proof}

Since $T^{\circ}$ is a noncompact manifold, we can apply the results of Church, Ellenberg, and Farb \cite{church2015fi} and Miller and Wilson \cite{miller2019higher} to study sequences of minimal generators of the homology groups of its ordered configuration space.

\begin{prop}\label{topgen}
For all $k\ge 3$ there are FB-modules $W$ dependent on $k$ such that,
\[
H_{k}(F(T^{\circ});\Q)=\bigoplus_{d=0}^{2k-2}M(W_{d}),
\]
and
\[
H_{0}(F(T^{\circ});\Q)=M(W_{0}), \indent H_{1}(F(T^{\circ});\Q)=\bigoplus_{d=0}^{1}M(W_{d}),\indent\text{and}\indent H_{2}(F(T^{\circ});\Q)=\bigoplus_{d=0}^{3}M(W_{d});
\]
moreover, in each case the top $W_{d}$ is nonzero.
\end{prop}

\begin{proof}
Since $T^{\circ}$ is noncompact, $H_{k}(F(T^{\circ}); \Q)$ is an FI\#-module. Theorem \ref{equiv} proves $H_{k}(F(T^{\circ}); \Q)$ is of the form $M(W)=\oplus_{d\ge 0}M(W_{d})$, for some FB-module $W$. Recall 
\[
M(W_{d})_{n}=W_{d}\otimes_{\Q[S_{d}]}\Q\cdot\text{Hom}_{FI}([d],[n]).
\]
From this we see that the dimension of $M(W_{d})_{n}$ is a polynomial in $n$ of degree $d$:
\[
\dim(M(W_{d})_{n})=\dim(W_{d})\cdot\binom{n}{d}.
\]
Therefore, the dimension of $M(W)_{n}$ is a polynomial in $n$ of degree $d$, with $d$ the largest integer such that $W_{d}\neq 0$. By Proposition \ref{thetheorem}, for $k\ge 3$, $b_{k}(F_{n}(T^{\circ})$ is a polynomial in $n$ of degree $2k-2$, and for $k=0,1,2$, $b_{k}(F_{n}(T^{\circ})$ is a polynomial in $n$ of degree $0, 1, 3$ respectively. Therefore, for $k\ge 3$,
\[
H_{k}(F(T^{\circ});\Q)=\bigoplus_{d=0}^{2k-2}M(W_{d}),
\]
and
\[
H_{0}(F(T^{\circ});\Q)=M(W_{0}), \indent H_{1}(F(T^{\circ});\Q)=\bigoplus_{d=0}^{1}M(W_{d}),\indent\text{and}\indent H_{2}(F(T^{\circ});\Q)=\bigoplus_{d=0}^{3}M(W_{d}),
\]
where the top degree term in each sum is nonzero.
\end{proof}

\begin{prop}\label{W1}
The FIM$^{+}$-module $\mathcal{W}_{1}^{T^{\circ}}$ is generated in degree $1$. 
\end{prop}

\begin{proof}
By Proposition \ref{topgen} if $k\ge 3$,
\[
H_{k}(F(T^{\circ});\Q)=\bigoplus_{d=0}^{2k-2}M(W_{d}),
\]
and
\[
H_{0}(F(T^{\circ});\Q)=M(W_{0}), \indent H_{1}(F(T^{\circ});\Q)=\bigoplus_{d=0}^{1}M(W_{d}),\indent\text{and}\indent H_{2}(F(T^{\circ});\Q)=\bigoplus_{d=0}^{3}M(W_{d});
\]
moreover, in each case the top $W_{d}$ is nonzero. By Theorem \ref{equiv} there is an equivalence $H_{0}^{\text{FI}}(M(W))\cong W$. Setting $k=\frac{n+1}{2}$, if $n\ge 5$, then $k\ge 3$ and
\[
H_{\frac{n+1}{2}}(F(T^{\circ}); \Q)=\bigoplus^{n-1}_{d=0}M(W_{d}),\indent\text{and}\indent H^{\text{FI}}_{0}\left(H_{\frac{n+1}{2}}(F(T^{\circ}); \Q)\right)_{n}=0.
\]
If $n=3$, then
\[
H_{2}(F(T^{\circ}); \Q)=\bigoplus^{2}_{d=0}M(W_{d}), \indent\text{and}\indent H^{\text{FI}}_{0}\left(H_{2}(F(T^{\circ}); \Q)\right)_{2}=W_{2}\neq0.
\]
Similarly, if $n=1$, then 
\[
H_{1}(F(T^{\circ}); \Q)=\bigoplus^{1}_{d=0}M(W_{d}), \indent\text{and}\indent H^{\text{FI}}_{0}\left(H_{1}(F(T^{\circ}); \Q)\right)_{1}=W_{1}\neq0.
\]

Thus, $\mathcal{W}_{1}^{T^{\circ}}(n)=H^{\text{FI}}_{0}(H_{\frac{n+1}{2}}(F(T^{\circ})))_{n}$ is nonzero only when $n=1,3$. Therefore, it suffices to check that $\mathcal{W}_{1}^{T^{\circ}}(3)$ is generated by $\mathcal{W}_{1}^{T^{\circ}}(1)$ as an FIM$^{+}$-module. 

Note that $\mathcal{W}_{1}^{T^{\circ}}(3)$ is a term in the arc resolution spectral sequence
\[
\mathcal{W}_{1}^{T^{\circ}}(3)=H_0^{\text{FI}}\Big(H_2(F(M))\Big)_{3}=E^{2}_{-1,2}[T^{\circ}](3).
\]

Since $\mathcal{W}_{1}^{T^{\circ}}(1)=H_0^{\text{FI}}\Big(H_{1}(F(M))\Big)_{1}$ and $E^{2}_{1,1}[T^{\circ}](3)=\text{Ind}^{S_{3}}_{S_{2}\times S_{1}}\mathcal{T}_{2}\boxtimes H_0^{\text{FI}}\Big(H_{1}(F(M))\Big)_{1}$, Proposition \ref{d2isFIM} implies $\mathcal{W}_{1}^{T^{\circ}}(3)$ is generated by $\mathcal{W}_{1}^{T^{\circ}}(1)$ as an FIM$^{+}$-module if
\[
d^{2}(E^{2}_{1,1}[T^{\circ}](3))=E^{2}_{-1,2}[T^{\circ}](3).
\]
Since 
\[
p+q+2 = -1+2+2 =3 \le 3 = n,
\]
Proposition \ref{arcresprop} implies $E^{\infty}_{-1,2}[T^{\circ}](3) =0$. As the arc resolution spectral sequence is zero in the lower half plane, it follows that if
\[
d^{3}:E^{3}_{2,0}[T^{\circ}](3)\to E^{3}_{-1,2}[T^{\circ}](3)
\]
is the zero map, then 
\[
d^{2}\left(E^{2}_{1,1}[T^{\circ}](3)\right)=E^{2}_{-1,2}[T^{\circ}](3)
\]
since for $r\ge 2$ the only entry of the spectral sequence with nonzero image under $d^{r}$ mapping into $E_{-1,2}^{r}[T^{\circ}](3)$ would be $E^{2}_{1,1}[T^{\circ}](3)$, and this map must be surjective as $E_{-1,2}^{\infty}[T^{\circ}](3)=0$. See Figure \ref{W1fig}. Since $E^{3}_{2,0}[T^{\circ}](3)= \mathcal{T}_{3}=\mathcal{L}_{3}$, Lemma \ref{bullseye} implies the $d^{3}$ map from $E^{3}_{2,0}[T^{\circ}](3)$ being the zero map is equivalent to the image of two particles orbiting a third clockwise in $T^{\circ}$ like a bullseye, as in Figure \ref{tripbull}, being homologically trivial in $H_{2}(F_{3}(T^{\circ}))$. 

\begin{figure}[H]
\centering
\begin{tikzpicture} \footnotesize
  \matrix (m) [matrix of math nodes, nodes in empty cells, nodes={minimum width=3ex, minimum height=5ex, outer sep=2ex}, column sep=3ex, row sep=3ex]{
 3    &  H_0^{\text{FI}}\Big(H_3(F(T^{\circ}))\Big)_{3}  &0 &  0& 0& \\  
 2    &  H_0^{\text{FI}}\Big(H_2(F(T^{\circ}))\Big)_{3} & 0&  0    & 0&  \\          
1     &  0   &0    & \text{Ind}_{S_{2}\times S_{1}}^{S_3} \mathcal{T}_{2} \boxtimes H_0^{\text{FI}} (H_1( F(T^{\circ})))_{1}  &0& \\             
 0     &  0  & 0 & 0  & \mathcal{T}_{3}& \\       
 \quad\strut &   -1  &  0  &  1  & 2  &\\}; 

 \draw[thick] (m-1-1.east) -- (m-5-1.east) ;
 \draw[thick] (m-5-1.north) -- (m-5-5.north east) ;
 
 \draw[-stealth, red] (m-3-4) -- (m-2-2)[midway,above] node [midway,below] {$d^2$} ;
\draw[-stealth, green] (m-4-5) -- (m-2-2) [midway,above] node [midway,below] {$d^3$};
 
\end{tikzpicture}
\caption{The $E^{2}[T^{\circ}](3)$-page of the arc resolution spectral sequence}
\label{W1fig}
\end{figure}
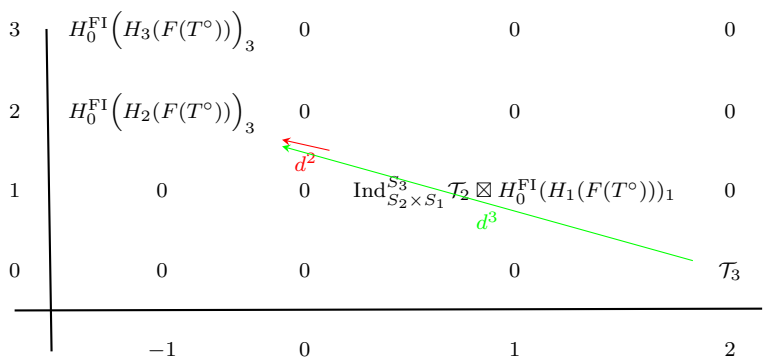

\begin{figure}[H]
\centering
\includegraphics[width = 8cm]{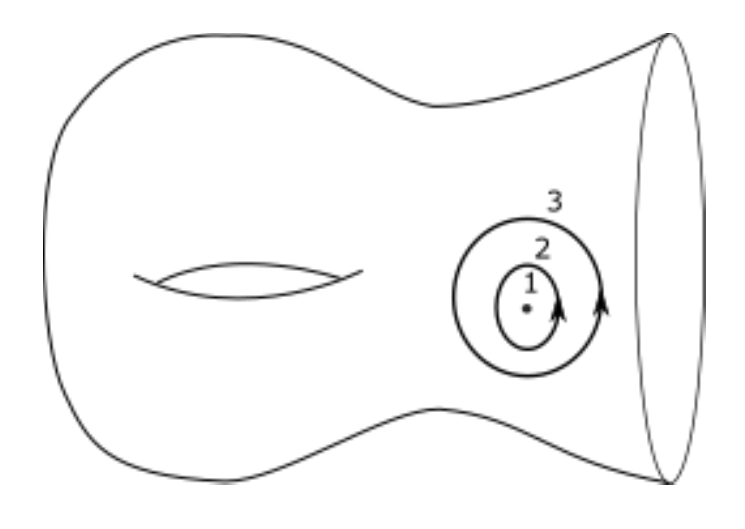}
\caption{A representative of the homology class of the image of $d^{3}:E^{3}_{2,0}[T^{\circ}](3)\to E^{3}_{-1,2}[T^{\circ}](3)$}
\label{tripbull}
\end{figure}

To see that two particles orbiting a third in $F_{3}(T^{\circ})$ is homologically trivial, we show that there is an embedding of a $3$-manifold with boundary in $F_{3}(T^{\circ})$, whose boundary corresponds to two particles orbiting a third in $T^{\circ}$. We start by embedding $(S^{1})^{3}$ in $(T^{\circ})^{3}$. Identify $T^{\circ}$ with $(\R/\Z)^{2}$ where the puncture is located at $(0,0)$, and let $S^{1}$ be identified with $\R/\Z$. Consider the map
\[
f:(\R/\Z)^{3}\cong(S^{1})^{3}\to (T^{\circ})^{3}\cong\left((\R/\Z)^{2}\right)^{3}
\]
\[
f(a,b,c)=\left(\frac{1}{2}, a, \frac{1}{2}+\frac{1}{8}\cos(2\pi b), a+\frac{1}{8}\sin(2\pi b), c, \frac{1}{2}\right).
\]

The map $f$ does not extend to an embedding of $(S^{1})^{3}$ in $F_{3}(T^{\circ})$. This can be seen by overlaying the three copies of $T^{\circ}$ onto a single $T^{\circ}$, as in the left side of Figure \ref{obstructbull}. To extend $f$ to a map to $F_{3}(T^{\circ})$, excise the region defined by the parameters $\frac{1}{4}<a<\frac{3}{4}$, $0\le b< 1$, and $\frac{1}{4}<c<\frac{3}{4}$, as in the right side of Figure \ref{obstructbull}. This is a solid torus $S^{1}\times D^{2}$, where the $S^{1}$ arises from the parameter $0\le b< 1$, and $D^{2}$ is the region defined by the parameters $\frac{1}{4}<a<\frac{3}{4}$, $\frac{1}{4}<c<\frac{3}{4}$. The boundary of the resulting manifold is a 2-torus defined by the product of the boundary of the square given by restricting the parameters $\frac{1}{4}<a<\frac{3}{4}$, $\frac{1}{4}<c<\frac{3}{4}$ and the circle $0\le b< 1$. We show that this boundary corresponds to two particles orbiting a third particle counterclockwise, as in Figure \ref{boundarypic}.

\begin{figure}[H]
\centering
\includegraphics[width = 8cm]{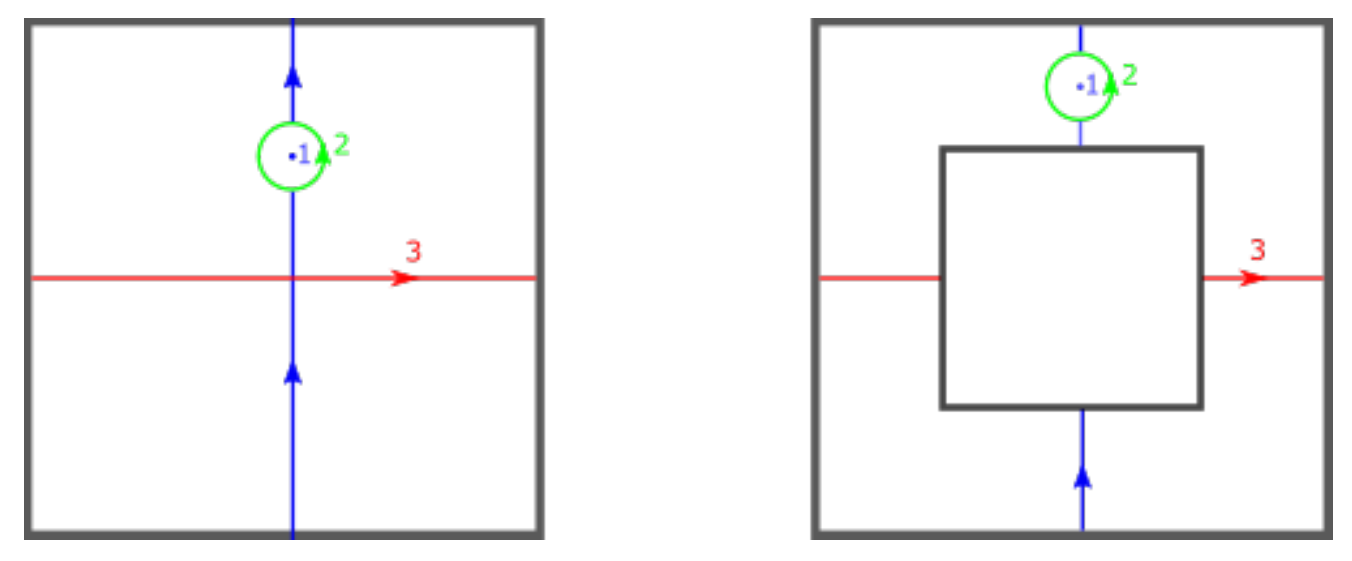}
\caption{Left: The projection onto $T^{\circ}$, where the puncture is at the corners, of the embedding of $(S^{1})^{3}$ into $(T^{\circ})^{3}$. Note that the paths of the particles intersect. Right: An extension of the embedding to $F_{3}(T^{\circ})$ given by excising a neighborhood around the intersection.}
\label{obstructbull}
\end{figure}

\begin{figure}[H]
\centering
\includegraphics[width = 8cm]{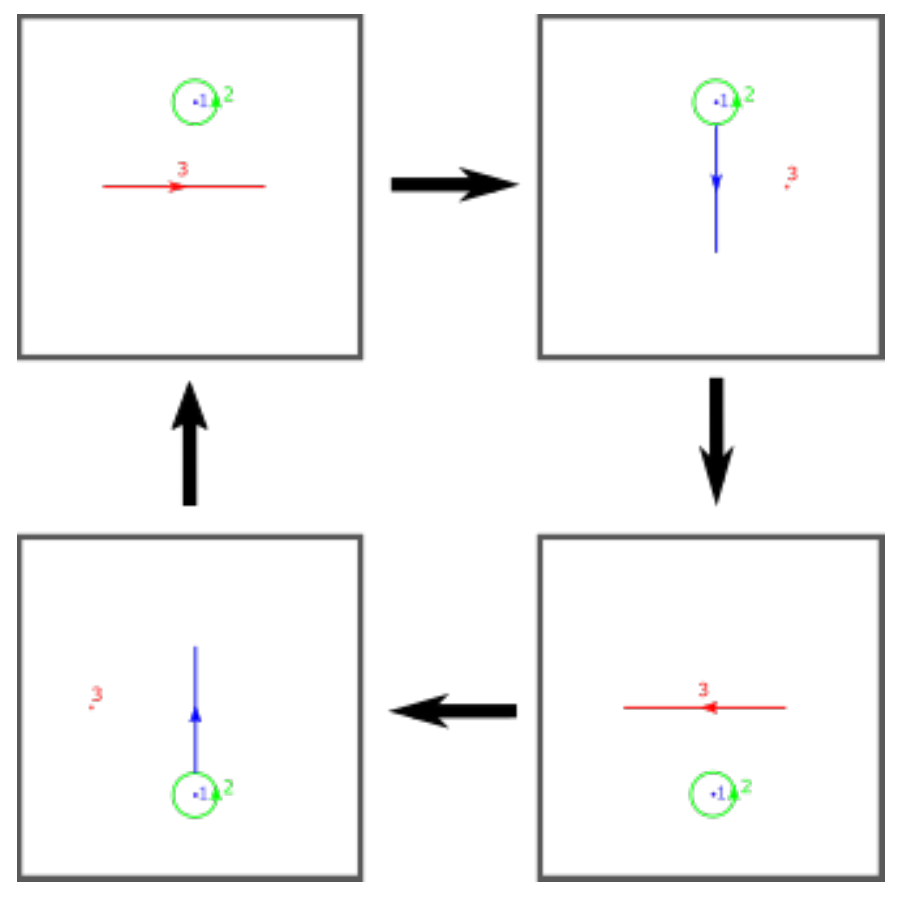}
\caption{The boundary of the 3-manifold in $F_{3}(T^{\circ})$}
\label{boundarypic}
\end{figure}

Fix $a=\frac{3}{4}$. The first particle is fixed at $(\frac{1}{2}, \frac{3}{4})$, the second particle freely orbits it in a circle given by $(\frac{1}{2}+\frac{1}{8}\cos(2\pi b), \frac{3}{4}+\frac{1}{8}\sin(2\pi b))$, and the third particle goes in a straight line from $(\frac{1}{4}, \frac{1}{2})$ to $(\frac{3}{4}, \frac{1}{2})$, i.e., underneath the orbiting pair from left to right. See the top left image in Figure \ref{boundarypic}.

Now fix $c=\frac{3}{4}$, and let $a$ vary from $\frac{3}{4}$ to $\frac{1}{4}$. This fixes the third particle at $(\frac{3}{4}, \frac{1}{2})$, and sends the first particle from $(\frac{1}{2}, \frac{3}{4})$ to $(\frac{1}{2}, \frac{1}{4})$ along a straight line, and the second particle freely orbits the first while following it. This corresponds to the orbiting pair going going to the left of the third particle from top to bottom, as in the top right of Figure \ref{boundarypic}.

Fix the first particle at $(\frac{1}{2}, \frac{1}{4})$, while the second freely orbits it along $(\frac{1}{2}+\frac{1}{8}\cos(2\pi b), \frac{1}{4}+\frac{1}{8}\sin(2\pi b))$. Let the third particle moves from $(\frac{3}{4}, \frac{1}{2})$ to $(\frac{1}{4}, \frac{1}{2})$ in a straight line. This corresponds to the third particle going above the orbiting pair from right to left, as in the bottom right of Figure \ref{boundarypic}.

Finally we finish the square, by fixing $c=\frac{1}{4}$, i.e., holding the third particle at $(\frac{1}{4}, \frac{1}{2})$. The first particle moves up by letting $a$ go from $\frac{1}{4}$ to $\frac{3}{4}$, while the second particle freely orbits it. This corresponds to sending the orbiting pair from bottom to top to the right of the third particle. See the bottom left of Figure \ref{boundarypic}.

This boundary corresponds to the third particle orbiting the orbiting pair, as in Figure \ref{tripbull}. Since this is the boundary of a 3-manifold in $F_{3}(T^{\circ})$, it is homologically trivial in $H_{2}(F_{3}(T^{\circ}))$. It corresponds to the image of $\mathcal{T}_{3}=\mathcal{L}_{3}$ under $d^{3}$, so \[
d^{3}(E^{3}_{2,0}[T^{\circ}](3))=0.
\]
Therefore,
\[
d^{2}\left(E^{2}_{1,1}[T^{\circ}](3)\right)=E^{2}_{-1,2}[T^{\circ}](3),
\] 
proving that $\mathcal{W}_{1}^{T^{\circ}}(3)$ is generated by $\mathcal{W}_{1}^{T^{\circ}}(1)$ as an FIM$^{+}$-module, and $\mathcal{W}_{1}^{T^{\circ}}(n)$ is generated in degree $1$.
\end{proof}

Now we prove a partial FIM$^{+}$-module version of Theorem \ref{mytheorem}.

\begin{thm}\label{genThm}
The sequence of minimal generators
\[
\mathcal{W}^{T^{\circ}}_{2}(n)=H^{\emph{FI}}_{0}\left(H_{\frac{n+2}{2}}(F(T^{\circ}); \Q)\right)_{n}
\]
is finitely generated as an FIM$^{+}$-module in degree $\le 4$; moreover, it is stably nonzero.
\end{thm}

\begin{proof}
If $k\ge 3$, then Proposition \ref{topgen} proves
\[
H_{k}(F(T^{\circ});\Q)=\bigoplus_{d=0}^{2k-2}M(W_{d}),
\]
and $W_{2k-2}$ is nonzero. By Theorem \ref{equiv} there is an equivalence $H_{0}^{\text{FI}}(M(W))\cong W$. Thus, if $n$ is an even number bigger than $3$ and we set $k=\frac{n+2}{2}$, then $k\ge 3$, $2k-2=2\left(\frac{n+2}{2}\right)-2=n$, and
\[
\mathcal{W}^{T^{\circ}}_{2}(n)=H^{\text{FI}}_{0}\left(H_{\frac{n+2}{2}}(F(T^{\circ}); \Q)\right)_{n} \cong W_{n}\neq 0.
\]
Therefore, for even $n\ge 3$, $\mathcal{W}^{T^{\circ}}_{2}(n)$ is nonzero, and the FIM$^{+}$-module $\mathcal{W}^{T^{\circ}}_{2}(n)$ is stably nonzero.

Next, we show that $\mathcal{W}_{2}^{T^{\circ}}$ is generated in degree at most $6$. The $E^{2}$-page of the arc resolution spectral sequence for the once-punctured torus is of the form
\[
E^{2}_{p,q}[T^{\circ}](n)=\text{Ind}^{S_{n}}_{S_{p+1}\times S_{n-p-1}}\mathcal{T}_{p+1}\boxtimes H^{\text{FI}}_{0}\big(H_{q}(F(T^{\circ})))_{n-p-1},
\]

so $\mathcal{W}_{2}^{T^{\circ}}(n)$ is an element of the $-1^{\text{st}}$-column of the $E^{2}$-page:
\[
\mathcal{W}_{2}^{T^{\circ}}(n) = H^{\text{FI}}_{0}\left(H_{\frac{n+2}{2}}(F(T^{\circ}))\right)_{n} = \text{Ind}^{S_{n}}_{S_{0}\times S_{n}}\mathcal{T}_{0}\boxtimes H^{\text{FI}}_{0}\big(H_{\frac{n+2}{2}}(F(T^{\circ})))_{n} =  E^{2}_{-1,\frac{n+2}{2}}[T^{\circ}](n).
\]

We will show that for $r\ge 2$ the only differential $d^{r}$ with nonzero image in $E^{r}_{-1,\frac{n+2}{2}}[T^{\circ}](n)$ is $d^{2}$, and that this differential is surjective; see Figure \ref{e2n8}. For $n\ge 8$, and $0\le q\le \frac{n-2}{2}$, consider $E^{2}_{p,q}[T^{\circ}](n)$ such that $p+q=\frac{n+2}{2}$, i.e., the slope $-1$ diagonal starting at, but not including, $E^{2}_{1,\frac{n}{2}}[T^{\circ}](n)$. The differentials of the spectral sequence will map the subquotients of these elements of the spectral sequence into $E_{-1,\frac{n+2}{2}}[T^{\circ}](n)$, and elements on this diagonal are of the form
\[
E^{2}_{\frac{n+2}{2}-q, q}[T^{\circ}](n)=\text{Ind}^{S_{n}}_{S_{\frac{n+2}{2}-q+1}\times S_{\frac{n-2}{2}+q-1}}\mathcal{T}_{\frac{n+2}{2}-q+1}\boxtimes H^{\text{FI}}_{0}\left(H_{q}(F(T^{\circ}))\right)_{\frac{n-2}{2}+q-1}.
\]
For $q\ge 3$,
\[
H_{q}(F(T^{\circ}))=\bigoplus_{d = 0}^{2q-2}M(W_{d}).
\]
Since $n\ge 8$ and $3\le q\le \frac{n-2}{2}$, it follows that $2q-2 < \frac{n-2}{2}+q-1$, so the $W_{\frac{n-2}{2}+q-1}$-term in the decomposition is $0$. Additionally, recall that
\[
H_{0}(F(T^{\circ}))=M(W_{0}),\indent H_{1}(F(T^{\circ}))=\bigoplus_{d = 0}^{1}M(W_{d}),\indent \text{and} \indent H_{2}(F(T^{\circ}))=\bigoplus_{d = 0}^{3}M(W_{d});
\]
since $n\ge 8$, we have $3<\frac{n-2}{2}+2-1$, $1< \frac{n-2}{2}+1-1$, and $0<\frac{n-2}{2}+0-1$, so the $W_{\frac{n-2}{2}+q-1}$-term in the decomposition is $0$ in these cases too. This proves that for $n\ge 8$, and $0\le q\le \frac{n-2}{2}$
\[
H^{\text{FI}}_{0}\left(H_{q}(F(T^{\circ}))\right)_{\frac{n-2}{2}+q-1}=0.
\]
Thus, for $n\ge 8$,  $0\le q\le \frac{n-2}{2}$, and $r\ge 2$, $E^{r}_{\frac{n+2}{2}-q, q}[T^{\circ}](k)=0$.

\begin{figure}[H]
\centering
\begin{tikzpicture} \footnotesize
  \matrix (m) [matrix of math nodes, nodes in empty cells, nodes={minimum width=3ex, minimum height=5ex, outer sep=2ex}, column sep=3ex, row sep=3ex]{
  \frac{n+2}{2}    & H_0^{\text{FI}}\Big(H_{\frac{n+2}{2}}(F(T^{\circ}))\Big)_{n}  & 0 &  \text{Ind}_{S_{2}\times S_{n-2}}^{S_n} \mathcal{T}_{2} \boxtimes H_0^{\text{FI}} (H_{\frac{n+2}{2}}( F(T^{\circ})))_{n-2}   &\text{Ind}_{S_{3}\times S_{n-3}}^{S_{n}} \mathcal{T}_{3} \boxtimes H_0^{\text{FI}} (H_{\frac{n+2}{2}}( F(T^{\circ})))_{n-3} &  \\  
 \frac{n}{2}    &  0  &0 &  \text{Ind}_{S_{2}\times S_{n-2}}^{S_n} \mathcal{T}_{2} \boxtimes H_0^{\text{FI}} (H_{ \frac{n}{2}} (F(T^{\circ})))_{n-2}   &\text{Ind}_{S_{3}\times S_{n-3}}^{S_n} \mathcal{T}_{3} \boxtimes H_0^{\text{FI}} (H_{ \frac{n}{2}} (F(T^{\circ})))_{n-3} &\\  
\frac{n-2}{2}    &  0 & 0&  0    &0 &  \\          
\frac{n-4}{2}     &  0   &0    & 0  &0 & \\             
 \frac{n-6}{2}    &  0  & 0 & 0  &0  & \\       
 \quad\strut &   -1  &  0  &  1  & 2  &\\}; 

\draw[-stealth, red] (m-2-4) -- (m-1-2)[midway,above] node [midway,below] {$d^2$} ;
\draw[-stealth, green] (m-3-5) -- (m-1-2) [midway,above] node [midway,below] {$d^3$};

 \draw[thick] (m-1-1.east) -- (m-6-1.east) ;
 \draw[thick] (m-6-1.north) -- (m-6-6.north east) ;
\end{tikzpicture}
\caption{$E^{2}_{p,q}[T^{\circ}](n)$ for $n\ge 8$}
\label{e2n8}
\end{figure}
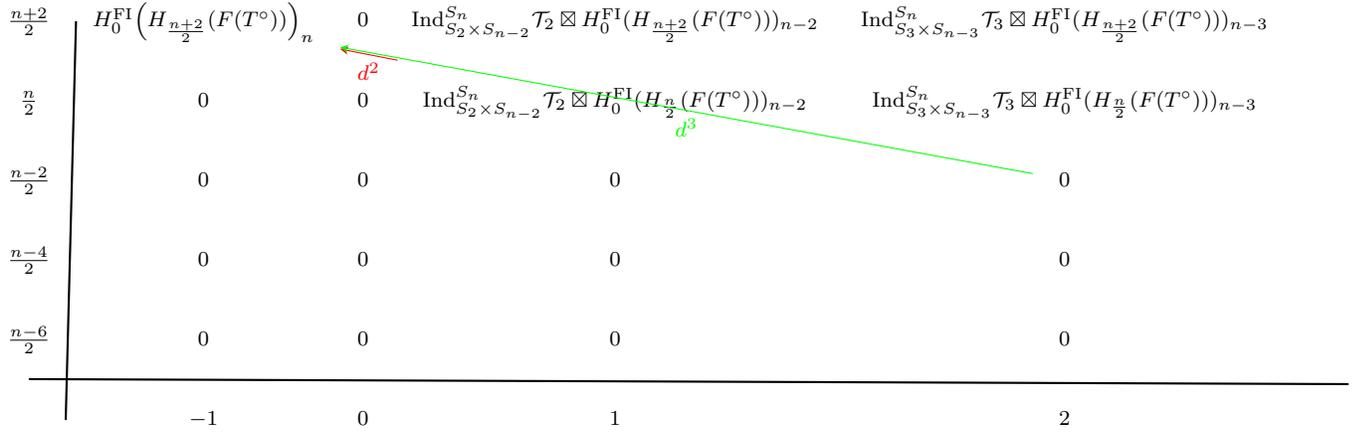

Thus, for $r\ge 2$, the only differential, $d^{r}$, that can have non-zero image in $E^{r}_{-1,\frac{n+2}{2}}[T^{\circ}](n)$ is $d^{2}$, as this is the only differential with nonzero domain. Since $n\ge 8$, Proposition \ref{arcresprop} proves $E_{-1, \frac{n+2}{2}}^{\infty}[T^{\circ}](n) = 0$ as $p+q+2=-1+\frac{n+2}{2}+2\le n$. Therefore,
\[
d^{2}\left(E^{2}_{1,\frac{n-2}{2}}[T^{\circ}](n)\right)= E^{2}_{-1,\frac{n+2}{2}}[T^{\circ}](n).
\]
Proposition \ref{d2isFIM} proves that $d^{2}$ from the $1^{\text{st}}$ column to the $-1^{\text{st}}$ column corresponds to the FIM$^{+}$-module morphism induced from $[n-2]\hookrightarrow[n]$. Since $d^{2}$ is surjective for all $n\ge 8$, $\mathcal{W}_{2}^{T^{\circ}}(n)$ is generated as a FIM$^{+}$-module by $\mathcal{W}_{2}^{T^{\circ}}(n-2)$.

Now we prove that $\mathcal{W}_{2}^{T^{\circ}}(6)$ is generated as an FIM$^{+}$-module from $\mathcal{W}_{2}^{T^{\circ}}(4)$. We first show that $E^{2}_{-1, 4}[T^{\circ}](6)$ is nonzero.

Since 
\[
E^{2}_{-1, 4}[T^{\circ}](6)=H^{\text{FI}}_{0}\left(H_{4}(F(T^{\circ}))\right)_{6}=\mathcal{W}_{2}^{T^{\circ}}(6),
\]
and by Proposition \ref{topgen},
\[
H_{4}(F(T^{\circ}))=\bigoplus_{d = 0}^{6}M(W_{d}),\indent\text{with}\indent W_{6}\neq0
\]
it follows that $E^{2}_{-1, 4}[T^{\circ}](6)$ is nonzero. Moreover, Proposition \ref{topgen} proves that
\[
H_{1}(F(T^{\circ}))=\bigoplus_{d=0}^{1}M(W_{d})\indent\text{and}\indent H_{0}(F(T^{\circ}))=M(W_{0}).
\]
Therefore,
\[
H^{\text{FI}}_{0}\left(H_{1}(F(T^{\circ}))\right)_{2}=0\indent\text{and}\indent H^{\text{FI}}_{0}\left(H_{0}(F(T^{\circ}))\right)_{1}=0,
\]
so
\[
E^{2}_{3, 1}[T^{\circ}](6)=\text{Ind}^{S_{6}}_{S_{4}\times S_{2}}\mathcal{T}_{4}\boxtimes H^{\text{FI}}_{0}\left(H_{1}(F(T^{\circ}))\right)_{2}=0,
\]
and 
\[
E^{2}_{4, 0}[T^{\circ}](6)=\text{Ind}^{S_{6}}_{S_{5}\times S_{1}}\mathcal{T}_{5}\boxtimes H^{\text{FI}}_{0}\left(H_{0}(F(T^{\circ}))\right)_{1}=0.
\]
Thus for $r\ge 2$,
\[
E^{r}_{3, 1}[T^{\circ}](6)=0,
\]
and
\[
E^{r}_{4, 0}[T^{\circ}](6)=0.
\]
If $E^{3}_{2,2}[T^{\circ}](6)=0$, then, for $r\ge 3$, all differentials $d^{r}$ into $E^{r}_{-1, 4}[T^{\circ}](6)$ have trivial image. Then, as $p+q+2=-1+4+2\le6=n$, Proposition \ref{arcresprop} would show that
\[
d^{2}\left(E^{2}_{1,3}[T^{\circ}](6)\right)=E^{2}_{-1, 4}[T^{\circ}](6).
\]
This, along with Proposition \ref{d2isFIM}, which states that the differential $d^{2}$ from the $1^{\text{st}}$ column to the $-1^{\text{st}}$ column corresponds to FIM$^{+}$-morphism induced from $[4]\hookrightarrow [6]$, would prove that $\mathcal{W}_{2}^{T^{\circ}}(6)$ is generated as an FIM$^{+}$-module from $\mathcal{W}_{2}^{T^{\circ}}(4)$.

We will show that the subset of $E^{2}_{4,1}[T^{\circ}](6)$, given by $t^{2}\left(E^{2}_{2,0}[\R^{2}](3)\otimes E^{2}_{1,1}[T^{\circ}](3)\right)$ surjects onto $E^{2}_{2,2}[T^{\circ}](6)$ under $d^{2}$, proving that $E^{3}_{2,2}[T^{\circ}](6)=0$.

The $0^{\text{th}}$-column of the $E^{2}$-page of the arc resolutions spectral sequence is zero, so the $d^{2}$-differential to the $0^{\text{th}}$-column of the arc resolution spectral sequence, 
\[
d^{2}:E^{2}_{2,0}[M](3)\to E^{2}_{0,1}[M](3)=0,
\]
must be the zero map for all manifolds $M$. In particular
\[
d^{2}:E^{2}_{2,0}[\R^{2}](3)\to E^{2}_{0,1}[\R^{2}](3)
\]
is the zero map. Proposition \ref{W1} proves that 
\[
d^{2}:E^{2}_{1,1}[T^{\circ}](3)\to E^{2}_{-1,2}[T^{\circ}](3)
\]
is surjective. By Lemma \ref{Leibniz} the map $t^{2}$ satisfies a Leibniz rule, so
\[
d^{2}\left(t^{2}(E^{2}_{2,0}[\R^{2}](3)\otimes E^{2}_{1,1}[T^{\circ}](3))\right)=t^{2}(E^{2}_{2,0}[\R^{2}](3)\otimes E^{2}_{-1,2}[T^{\circ}](3)).
\]
Since $E^{2}_{2,2}[T^{\circ}](6)=\text{Ind}^{S_{6}}_{S_{3}\times S_{3}} \mathcal{T}_{3}\boxtimes H_{0}^{\emph{FI}}(H_{2}(F(T^{\circ})))_{3}$, $E^{2}_{2,0}[\R^{2}](3)= \mathcal{T}_{3}$, and $E^{2}_{-1,2}[T^{\circ}](3) = H_{0}^{\emph{FI}}(H_{2}(F(T^{\circ})))_{3}$, it follows from the definition of $t$ that
\[
t^{2}\left(E^{2}_{2,0}[\R^{2}](3)\otimes E^{2}_{-1,2}[T^{\circ}](3)\right)= E^{2}_{2,2}[T^{\circ}](6).
\]
Therefore,
\[
d^{2}\left(E^{2}_{4,1}[T^{\circ}](6)\right)= E^{2}_{2,2}[T^{\circ}](6),
\]
and $E^{3}_{2,2}[T^{\circ}](6) =0$. Thus, $d^{2}$ is the only nontrivial differential for $r\ge 2$ into $E^{r}_{-1, 4}[T^{\circ}](6)$, and $\mathcal{W}_{2}^{T^{\circ}}(6)$ is generated as an FIM$^{+}$-module from $\mathcal{W}_{2}^{T^{\circ}}(4)$. It follows that $\mathcal{W}_{2}^{T^{\circ}}(n)$ is generated in degree at most $4$ as an FIM$^{+}$-module.
\end{proof}

\begin{cor}
The sequence of minimal generators $\mathcal{W}_{2}^{T^{\circ}}$ is not a free FIM$^{+}$-module.
\end{cor}

\begin{proof}

Theorem \ref{genThm} proves $\mathcal{W}_{2}^{T^{\circ}}$ is generated in degree at most $4$ and that $\mathcal{W}_{2}^{T^{\circ}}(4)\neq 0$. It follows from the definition of ordered configuration space that
\[
H_{1}(F_{0}(T^{\circ}),\Q)=0,
\]
so
\[
\mathcal{W}_{2}^{T^{\circ}}(0)=H_{0}^{\emph{FI}}\left(H_{1}(F(T^{\circ}),\Q)\right)_{0}=0.
\]
Proposition \ref{thetheorem} proves that
\[
H_{2}(F(T^{\circ}),\Q)_{1}=0\indent\text{and}\indent H_{2}(F(T^{\circ}),\Q)_{2}\neq0,
\]
so
\[
\mathcal{W}_{2}^{T^{\circ}}(2)=H_{0}^{\emph{FI}}\left(H_{2}(F(T^{\circ}),\Q)\right)_{2}\neq0.
\]
Therefore, to show that $\mathcal{W}_{2}^{T^{\circ}}$ is not free, it suffices to prove that $M^{\text{FIM}^{+}}(\mathcal{W}_{2}^{T^{\circ}}(2))\not\subseteq \mathcal{W}_{2}^{T^{\circ}}$.

Proposition \ref{thetheorem} proves that $\dim(\mathcal{W}_{2}^{T^{\circ}}(2))=5$, so 
\[
\dim(M^{\text{FIM}^{+}}(W_{2})_{4}) = \binom{4}{2}\times \dim(M^{\text{FIM}^{+}}(0))_{4-2} \times\dim(W_{2})= 6\times 1\times 5=30.
\]
Since, $\dim(\mathcal{W}_{2}^{T^{\circ}}(2)) = 14 <30$ it follows that $M^{\text{FIM}^{+}}(W_{2})_{4}\not\subseteq \mathcal{W}_{2}^{T^{\circ}}(4),$ and $\mathcal{W}_{2}^{T^{\circ}}$ is not a free FIM$^{+}$-module.
\end{proof}

This corollary, along with Theorem \ref{genThm}, provide an FIM$^{+}$-module version of Theorem \ref{mytheorem}. Moreover, they demonstrate the first example of a non-free, stably nonzero secondary representation stability sequence in positive genus, This is the first stably nonzero example for any manifold for which the homology groups of the ordered configuration spaces were not already explicitly known.

\bibliographystyle{amsalpha}
\bibliography{PuncturedTorusBib}

\providecommand{\bysame}{\leavevmode\hbox to3em{\hrulefill}\thinspace}
\providecommand{\MR}{\relax\ifhmode\unskip\space\fi MR }
\providecommand{\MRhref}[2]{%
  \href{http://www.ams.org/mathscinet-getitem?mr=#1}{#2}
}
\providecommand{\href}[2]{#2}
\begin{thebibliography}{CEF15}

\bibitem[CEF15]{church2015fi}
Thomas Church, Jordan~S Ellenberg, and Benson Farb, \emph{{FI}-modules and
  stability for representations of symmetric groups}, Duke Mathematical Journal
  \textbf{164} (2015), no.~9, 1833--1910.

\bibitem[Chu12]{church2012homological}
Thomas Church, \emph{Homological stability for configuration spaces of
  manifolds}, Inventiones mathematicae \textbf{188} (2012), no.~2, 465--504.

\bibitem[Coh10]{cohen2010introduction}
Frederick~R Cohen, \emph{Introduction to configuration spaces and their
  applications}, 2010.

\bibitem[KM15]{kupers2015improved}
Alexander Kupers and Jeremy Miller, \emph{Improved homological stability for
  configuration spaces after inverting 2.}, Homology, Homotopy \& Applications
  \textbf{17} (2015), no.~1.

\bibitem[MW19]{miller2019higher}
Jeremy Miller and Jennifer Wilson, \emph{Higher-order representation stability
  and ordered configuration spaces of manifolds}, Geometry \& Topology
  \textbf{23} (2019), no.~5, 2519--2591.

\bibitem[NSS19]{nagpal2019noetherianity}
Rohit Nagpal, Steven~V Sam, and Andrew Snowden, \emph{Noetherianity of some
  degree two twisted skew-commutative algebras}, Selecta Mathematica
  \textbf{25} (2019), no.~1, 4.

\bibitem[Pag20]{pagaria2020asymptotic}
Roberto Pagaria, \emph{Asymptotic growth of betti numbers of ordered
  configuration spaces on an elliptic curve}, arXiv preprint arXiv:2005.02106
  (2020).

\bibitem[RW13]{randal2013homological}
Oscar Randal-Williams, \emph{Homological stability for unordered configuration
  spaces}, The Quarterly Journal of Mathematics \textbf{64} (2013), no.~1,
  303--326.

\bibitem[Sin06]{sinha2006homology}
Dev Sinha, \emph{The homology of the little disks operad}, arXiv preprint
  math/0610236 (2006).

\bibitem[SS12]{sam2012introduction}
Steven~V Sam and Andrew Snowden, \emph{Introduction to twisted commutative
  algebras}, arXiv preprint arXiv:1209.5122 (2012).

\bibitem[Wil18]{wilson2018introduction}
Jenny Wilson, \emph{An introduction to {FI}--modules and their
  generalizations}, URL: http://www. math. lsa. umich. edu/\~{}
  jchw/FILectures. pdf (2018).

\end{thebibliography}

\end{document}